\documentclass[twoside,a4paper,12pt]{article}
\usepackage{amssymb, amsmath, amsthm}
\usepackage{fancyhdr}
\usepackage{color}

\newdimen\myMargin
\textwidth \paperwidth
\advance\textwidth -2\myMargin
\textheight \paperheight
\advance\textheight -2\myMargin
\advance\oddsidemargin \myMargin
\advance\evensidemargin \myMargin
\advance\topmargin \myMargin
\advance\textheight -17 true mm
\advance\topmargin 12.5 true mm

\newcommand{\A}{{\cal A}}

\newcommand{\K}{\mathbb K}

\newcommand{\R}{\mathbb R}
\newcommand{\Q}{\mathbb Q}
\newcommand{\C}{\mathbb C}
\newcommand{\OK}{\mathcal O}
\newcommand{\Aalg}{\mathcal A}

\newcommand{\Z}{\mathbb Z}

\newcommand{\rank}{\mathrm{rank}}

\newtheorem{theorem}{Theorem}
\newtheorem{lemma}[theorem]{Lemma}
\newtheorem{corollary}[theorem]{Corollary}
\newtheorem{proposition}[theorem]{Proposition}

\theoremstyle{remark}
\theoremstyle{remark}

\title{Improved algorithms for splitting full matrix algebras}

\author{\normalsize
  \begin{minipage}{0.3\linewidth}
    \large
    G\'abor Ivanyos \\
    \footnotesize
    Computer and Automation Research 
Institute, Hungarian Acad. Sci. \\
    \texttt{Gabor.Ivanyos@sztaki.hu} \\
    \normalsize
  \end{minipage}
  \qquad
  \begin{minipage}{0.3\linewidth}
    \large
    \'Ad\'am D. Lelkes \\
    \footnotesize
    Dept. of Algebra, Budapest 
    Univ. of Technology and Economics \\
    \texttt{lelkesa@math.bme.hu}
    \normalsize
  \end{minipage}
  \qquad
  \begin{minipage}{0.3\linewidth}
    \large
    Lajos R\'onyai \\
    \footnotesize
    Computer and Automation Research
     Institute, Hungarian Acad. Sci. \\
    Dept. of Algebra, Budapest 
    Univ. of Technology and Economics \\
    \texttt{lajos@ilab.sztaki.hu}
    \normalsize
  \end{minipage}
  \vspace{0.5cm}
}

\begin{document}

\thispagestyle{empty}

\maketitle

\footnotetext{
Our work was supported in part by OTKA grants 
NK 105645, K77476, and K77778. 
}


\begin{abstract}
Let $\K$ be an algebraic number field of degree $d$ and discriminant
$\Delta$
over $\Q$. Let 
$\A$ be an associative algebra over $\K$ given by structure constants such
that $\A\cong M_n(\K)$ holds for some positive integer $n$. Suppose that
$d$, $n$ and $|\Delta|$ are bounded. In a previous paper a polynomial time 
ff-algorithm was given to construct explicitly an isomorphism
$\A \rightarrow M_n(\K)$. 

Here we simplify and improve this algorithm in the cases $n\leq 43$,
$\K=\Q$, and $n=2$, with $\K=\Q(\sqrt{-1})$ or 
$\K=\Q(\sqrt{-3})$. The improvements are based on work by Y. Kitaoka and
R. Coulangeon on tensor products of lattices. 
\end{abstract}

\section{Introduction}

The following {\em explicit isomorphism problem} is important in
computational representation theory: 
let $\K$ be an algebraic number field, $\A$ an associative algebra
over $\K$. Suppose that
$\A$ is isomorphic to the full matrix algebra $M_n(\K)$.
Construct explicitly an isomorphism $\A\rightarrow M_n(\K)$. Or,
equivalently, give an irreducible $\A$ module.

The algebra $\A$ is considered to be given by a collection of {\em structure 
constants} $\gamma_{ijk}\in
\K$. They form the multiplication table of $\A$ with respect to some $\K$
basis $a_1, \ldots ,a_m$:
the products $a_ia_j$ can be expressed as
$$ a_ia_j=\gamma _{ij1}a_1+\gamma _{ij2}a_2+\cdots +\gamma _{ijm}a_m. $$

In \cite{Ronyai} a polynomial time ff-algorithm was given for the case of
the problem, when $n$ and the degree and the discriminant of $\K$ are all
bounded. Applications were also outlined there, including some
parametrization problems of algebraic geometry.  The methods of
\cite{Ronyai} are based on theoretical and algorithmic results on lattices
and eventually boil down to enumerating short vectors in  
some lattices in real Euclidean spaces. 

In this paper we present considerable improvements of the methods of
\cite{Ronyai} in the case when the ground field $\K$ is the rationals, and 
$n\leq 43$. This is based on results of Kitaoka \cite{Kitaoka},
\cite{Kitaoka2} on tensor products of lattices, in particular we make use of
the powerful result of 
Corollary \ref{Kitaoka_cor}. In Theorem \ref{kitconstant} we prove a
quantitative version, which allows further reduction in computing time.   
Some of Kitaoka's results
have been extended by Coulangeon \cite{Coulangeon} from $\Q$ to imaginary 
quadratic
fields. Using these we also obtain an improvement of the original algorithm
when $\K=\Q(\sqrt{-1})$ (Gaussian rationals) or $\K=\Q(\sqrt{-3})$
(Eisenstein numbers) for the case $n=2$.  The new algorithms are simpler
and faster than the original ones. 

For the basic definitions and facts on lattices in real Euclidean spaces
we refer to \cite{Conway}, \cite{Martinet}, and  \cite{Milnor}.

\section{Full matrix algebras over $\Q$ }

A {\em (full) lattice} $L\subset \R^n$ is the free Abelian group 
generated by  $n$ linearly independent vectors 
$\mathbf{b}_1,\mathbf{b}_2,\ldots, \mathbf{b}_n$.
Then $M=(\mathbf b_1 | \mathbf b_2 | \ldots | \mathbf b_n)$ is a 
matrix of $L$, and $|\det M|$ is called the {\em determinant} of 
$L$, and is denoted by $\det L$. We denote by $\lambda_1(L)$ the
Euclidean length of the shortest nonzero vector from $L$.
The $n$th \emph{Hermite's constant} is 
$$\gamma_n := \sup_L \left(\frac{\lambda_1(L)}{(\det L)^{1/n}}\right)^2,$$
where $L$ is a full lattice  in $\R^n$.
Hermite proved that $\gamma_n$ actually exists.
The exact value of $\gamma_n$ is known only for 
$n\in\{1,2,\ldots,8,24\}$. 

We briefly recall now the definition of the tensor product of lattices, for 
more information see section 1.10 of Martinet \cite{Martinet}, 
and section 7 in Kitaoka \cite{Kitaoka2}.
Let $L$ and $M$ be two lattices in $\mathbb R^m$ and $\mathbb R^n$, 
respectively. The tensor product of $\Z$ modules 
$L\otimes_{\mathbb Z} M$ embeds in the straightforward way 
into $\mathbb R^m\otimes_{\mathbb R} \mathbb R^n$.
This allows one to define $L\otimes M$ as
the set of integral linear combinations of the
tensors $\mathbf x\otimes \mathbf y$ from $\mathbb R^m\otimes_{\mathbb R}
\mathbb R^n$ where $\mathbf x\in L$ and
$\mathbf y\in M$.

Note that, in terms of coordinates, $L\otimes M$ can be viewed as the set
(actually lattice)
of $m$ by $n$ matrices over $\mathbb R$ which are integral linear
combinations
of dyads of the form $\mathbf x\mathbf y^T$, where $\mathbf x\in L$ and
$\mathbf y\in M$.
Note also that $\mathbb R^m\otimes \mathbb R^n$ is an Euclidean space with
the law
$\langle \mathbf x_1\otimes \mathbf y_1, \mathbf x_2\otimes \mathbf
y_2\rangle=
\langle \mathbf x_1,\mathbf x_2\rangle \langle \mathbf y_1,\mathbf
y_2\rangle$.
In this setting the norm on the tensor product 
$\mathbb R^m\otimes \mathbb R^n$
is actually the Frobenius norm on the space of matrices $M_{m,n}(\mathbb R)$. 

Let $L$ be a full lattice in $\R^m$. The {\em dual} $L^*$ of $L$ consists 
of those vectors $\mathbf y\in  \R^m$ for which we have 
$\langle \mathbf x, \mathbf y\rangle \in \Z$ holds for every $\mathbf x\in
L$. The supremum $\gamma'_n$ of $\lambda_1(L)\lambda_1(L^*)$ among  
full lattices $L\subset
\R^n$ is the {\em Berg\'e-Martinet constant}, see \cite{BM},
\cite{Martinet}. It is known that $\gamma'_n\leq \gamma_n$ for every $n$. 

Following Y. Kitaoka \cite{Kitaoka} we 
say that a lattice $L$ is of \emph{$E$-type} if every minimal nonzero vector of
$L\otimes M$ is of the form $\mathbf x\otimes \mathbf y$
($\mathbf x\in L$, $\mathbf y\in M$) for any lattice $M$.

Kitaoka proved in  \cite{Kitaoka}, and in Theorem 7.1.1. of \cite{Kitaoka2} 
the following.

\begin{theorem}[Kitaoka]\label{kitaoka1}
If $L$ is a lattice of rank at most $43$, then $L$ is of $E$-type.
\end{theorem}

We remark here that by a theorem of Steinberg (see  \cite{Milnor}, Chapter
II, \S 9) for any integer $n\geq 292$ 
there exists a lattice $L$ such that 
$\lambda_1(L\otimes L^*)<\lambda_1(L)\lambda_1(L^*)$.
Thus, the conclusion of Kitaoka's theorem does not hold for larger values of
$n$. 

We can  apply Kitaoka's theorem  to maximal orders of the form 
$\Lambda =QM_n(\Z) Q^{-1}\subset
M_n(\R)$, where $Q\in GL_n(\R)$. 

\begin{corollary} \label{Kitaoka_cor}
For any dimension $n\le 43$ and subring $\Lambda\subset M_n(\R )$ of the
above form the nonzero matrices with the smallest Frobenius norm in 
$\Lambda$ are of rank one.
\end{corollary}

\begin{proof}
The order $\Lambda$ is  given by the transformation matrix $Q$. As in the
proof of Theorem 1 of \cite{Ronyai}, we see that as a lattice
$\Lambda\cong Q\Z^n \otimes (Q\Z^n )^*$.  But $Q\Z^n$ is a rank $n$ lattice,
hence it is of $E$-type, giving that the matrices of minimal norm
in $\Lambda$
are dyadic products of vectors from $Q\Z^n$ and $(Q\Z^n)^*$, and hence 
have rank 1 as matrices from $M_n(\R)$. 
\end{proof}

Thus, when $n\leq 43$, then  the smallest zero divisors
in any $\Lambda$ have rank one. In particular, no matrix 
from $\Lambda$ of rank at least 2 can have minimal length. 
By using a modified variant of Kitaoka's original argument,
we prove a slightly stronger, quantitative version of the latter statement.

Recall that the {\em rank} of a nonzero tensor $\mathbf v\in L\otimes 
M$ is the smallest positive integer $r$ such that $\mathbf v$ can be 
written as 
\begin{equation} \label{repr}
\mathbf v=\sum_{i=1}^r \mathbf x_i\otimes \mathbf y_i
\end{equation}
for some 
$\mathbf x_1,\ldots, \mathbf x_r\in L$ and  $\mathbf y_1,\ldots, 
\mathbf y_r\in M$. 

\begin{theorem}\label{kitconstant}
Let $L$ and $M$ be lattices. Then for every tensor $\mathbf v\in L\otimes M$ 
of rank $r$ we have 
$$\lVert \mathbf v \rVert \geq  \sqrt{\frac{r}{\gamma_r^2}} 
\lambda_1(L\otimes M).$$
\end{theorem}

We remark that the above bound is meaningful when the dimension of $L$
or $M$ is at most $43$. Then obviously $r\leq 43$, and by 
Lemma 7.1.2 from \cite{Kitaoka2} for $2\leq r\leq 43$ we have  
$1< r/\gamma_r^2$. For $r$ large the bound becomes trivial, because 
$r/\gamma_r^2$ tends to zero as $r$ grows.  We shall need Lemma 7.1.3 from 
\cite{Kitaoka2}: 

\begin{lemma} \label{kitlemma}
Let $A,B\in M_n(\R)$ be positive definite real symmetric matrices.
Then we have $\mathrm{Tr}(AB)\ge n\root{n}\of{\det A}\root{n}\of{\det B}$.
\end{lemma}

\medskip

\noindent
{\em Proof of Theorem \ref{kitconstant}.}
Let $\mathbf v\in L\otimes M$ be a tensor of rank $r$. Then $\mathbf v$ 
can be written in the form (\ref{repr}).
Let $L_1$ be the lattice generated by 
$\{\mathbf x_1,\ldots,\mathbf x_r\}$, and similarly let $M_1$ be the lattice
spanned by $ \{\mathbf y_1,\ldots,\mathbf
y_r\}$. By the minimality of representation (\ref{repr}) the rank of these 
sublattices is $r$. Noting that 
$$\lVert \mathbf v\rVert^2=\left\lVert \sum_{i=1}^r \mathbf x_i\otimes
\mathbf y_i\right\rVert^2=
\sum_{i,j=1}^r \langle \mathbf x_i, \mathbf x_j\rangle \langle \mathbf y_i,
\mathbf y_j\rangle=
\mathrm{Tr}([\langle \mathbf x_i, \mathbf x_j\rangle]_{i,j=1}^r \cdot
[\langle \mathbf y_i,    
\mathbf y_j\rangle]_{i,j=1}^r),$$
and by using Lemma \ref{kitlemma} we obtain
\begin{equation} \label{lower}
\lVert \mathbf v\rVert^2 \ge r\left(\det [\langle \mathbf x_i, \mathbf
x_j\rangle]\cdot
\det [\langle \mathbf y_i, \mathbf y_j\rangle] \right)^{1/r}. 
\end{equation}
Now let us assume for contradiction 
that $\lVert \mathbf v\rVert^2<(r/\gamma_r^2) \lambda_1(L\otimes M)^2$.
It follows that 
$$\lVert \mathbf v\rVert^2<\frac{r}{\gamma_r^2} (\lambda_1(L)
\lambda_1(M))^2   
\le \frac{r}{\gamma_r^2} (\lambda_1(L_1)
\lambda_1(M_1))^2.$$
Combining this with (\ref{lower}) we obtain
$$r< \frac{r}{\gamma_r^2} \cdot \frac{\lambda_1(L_1)^2}
{(\det [\langle \mathbf x_i, \mathbf x_j\rangle])^{1/r}}
\cdot\frac{\lambda_1(M_1)^2}
{(\det [\langle \mathbf y_i, \mathbf y_j\rangle])^{1/r}}
\le \frac{r}{\gamma_r^2} \gamma_r^2,$$
since $[\langle \mathbf x_i, \mathbf x_j\rangle]_{i,j=1}^r$ and $[\langle
\mathbf y_i, \mathbf y_j\rangle]_{i,j=1}^r$   
are Gram matrices for $L_1$ and $M_1$, 
respectively. The contradiction  
finishes the proof. $\Box$

\medskip

Using the known values of $\gamma_r$ simple calculation gives that 
the minimal value of $r/\gamma_2^2$ for $2 \leq r\leq 8$ 
is $\frac 32$, which is attained at $r=2$. 
We remark that the bound of 
Theorem \ref{kitconstant} is sharp, at least for $r=2$. 
This is
demonstrated by the hexagonal lattice $A_2\leq \R^2$ which is generated by
the vectors 
$$\begin{pmatrix}\frac12   \vspace{2 mm} \\ \frac{\sqrt3}{2}
\end{pmatrix},~ 
\begin{pmatrix} 1  \vspace{2 mm} \\ 0\end{pmatrix}.$$
It is known that $A_2$ is attains the Hermite constant, moreover this holds
also for the dual lattice $A_2^*$ which is spanned by the vectors  
$$\begin{pmatrix} 0  \vspace{2 mm} \\ \frac{2}{\sqrt3}
\end{pmatrix},~
\begin{pmatrix} 1  \vspace{2 mm} \\ -\frac{1}{\sqrt3}\end{pmatrix}.$$
Some calculation shows that the minimal norm among the 
rank two tensors in $A_2\otimes A_2^*$ is 
$\sqrt2=\sqrt{\frac32}\cdot\frac{2}{\sqrt{3}}=
\sqrt{\frac 32}\lambda_1(A_2)\lambda_1(A_2^*)$.

\medskip
In \cite{Ronyai} it was shown that for the shortest nonzero matrix
$\mathbf v\in \Lambda=QM_n(\Z)Q^{-1}$  we have $\lVert \mathbf v \rVert\leq
\gamma_n$, where $\gamma_n$ is the Hermite constant. This can be
strengthened as follows. Using again that $\Lambda\cong Q\Z^n \otimes
(Q\Z^n )^*$, we obtain that 
\begin{equation} \label{berge-martinet}
\lVert \mathbf v \rVert\leq \lambda_1(Q\Z^n)\lambda_1((Q\Z^n)^*)\leq
\gamma'_n,
\end{equation}
where $\gamma'$ is the Berg\'e-Martinet constant.

\section{The modified IRS algorithm over $\Q$ for $n\leq 43$}

The input of the algorithm is an associative algebra $\A$ over $\Q$ given
by structure constants. It is known that $\A$ is isomorphic to the full
matrix algebra $M_n(\Q)$. The objective of the algorithm is to find an
element $C\in \A$ which has rank one, when viewed as a matrix from 
$M_n(\Q)$. 

The first four steps of the algorithm below are identical to the first four
steps of the corresponding algorithm from \cite{Ronyai}. The last two steps 
of that method are replaced here by a new step 5:

\begin{enumerate}

\item
Construct a maximal order $\Lambda $ in ${\cal A}$. 

\item
Compute an embedding of ${\mathcal A}$ into $M_n(\mathbb R)$. 
This way we have a Frobenius norm on  ${\mathcal A}$. For $X\in {\mathcal
A}$
we can set $\|X\|=\sqrt{Tr(X^T X)}$. Also, via this embedding $\Lambda $ can 
be viewed as a full lattice in $\mathbb R^m$, where $m=n^2$. The length
$\lVert\mathbf{v}\rVert$ of
a lattice vector $\mathbf{v}$ is just the Frobenius norm of $\mathbf{v}$ as
a matrix.

\item
Compute a rational approximation $A$ of our basis $B$ of $\Lambda$ with
a suitable precision.

\item 
Obtain a reduced basis $\mathbf{b}_1,\ldots ,\mathbf{b}_m$ of the lattice
$\Lambda\subset \mathbb R^m$
by computing an LLL-reduced basis from  $A$. 
The value  $c_m$ of reducedness is 
$\left(\gamma_m \right)^{\frac m2}\left( \frac 32
\right)^m
2^{\frac{m(m-1)}{2}}$.

\item

Generate all integral linear combinations 
$$C=\sum_{i=1}^m  \alpha_i\mathbf{b}_i,$$
where $\alpha_i$ are integers, 
$|\alpha_i|\le c_m$,
until a $C$ is found with $\rank\, C=1$. Output this $C$.

\end{enumerate}

\bigskip

\begin{theorem}
This algorithm is correct when $n\le 43$. 
Moreover, it runs in ff-polynomial time.
\end{theorem}

\begin{proof}
For the details and timing of steps 1-4 we refer to the proof of Theorem 1
from \cite{Ronyai}. As a result of these computations we obtain a basis
$ \mathbf{b}_1,\ldots  ,\mathbf{b}_m$ of the lattice $\Lambda$ such
that 
 $$\lVert\mathbf{b}_1\rVert\cdot \lVert\mathbf{b}_2\rVert\cdots
\lVert\mathbf{b}_m\rVert\leq c_m\cdot \det(\Lambda) $$
holds with 
 $$c_m=\left(\gamma_m \right)^{\frac m2}\left( \frac 32 \right)^m
2^{\frac{m(m-1)}{2}}.$$
We recall the following bound by H. W. Lenstra \cite{Lenstra}.

\begin{lemma} \label{coefficients}
Let $\Gamma$ be a full lattice in $\mathbb R^m$.
Suppose that we have a basis $\mathbf{b}_1,\ldots, \mathbf{b}_m$ of $\Gamma$ 
over
$\mathbb Z$ such that
 $$\lVert\mathbf{b}_1\rVert\cdot \lVert\mathbf{b}_2\rVert\cdots
\lVert\mathbf{b}_m\rVert\leq c\cdot \det(\Gamma) $$
holds for a real number $c>0$. Suppose that
$$\mathbf{v} =\sum_{i=1}^m\alpha_i \mathbf{b}_i \in \Gamma,~~~\alpha_i\in
\mathbb Z. $$
Then we have
$|\alpha_i|\leq c \frac{\lVert\mathbf{v}\rVert}
{\lVert\mathbf{b}_i\rVert}$ for $i=1,\ldots, m$.
\end{lemma}

Let $\mathbf{v}\in \Lambda$ be a nonzero vector with minimal length. Then,
on one hand, by Corollary \ref{Kitaoka_cor} $\mathbf v$ is a matrix of
rank one. On the other hand, when $\mathbf v$ is expressed as an integer 
linear combination $\mathbf{v} =\sum_{i=1}^m\alpha_i \mathbf{b}_i$ then
by Lemma \ref{coefficients} we have
\begin{equation} \label{alpha-bound}
|\alpha_i|\leq c_m \frac{\lVert\mathbf{v}\rVert}
{\lVert\mathbf{b}_i\rVert}\leq c_m.
\end{equation}
This implies that there is indeed a rank one matrix $C\in \Lambda$ among the 
linear combinations enumerated. 

To obtain the timing bound, we observe that at step 5 we enumerate at most 
$(2c_m+1)^m$ linear combinations, and this value is bounded by our assumption
$n\leq 43$.
\end{proof}

\medskip
We remark that the upper bound $|\alpha_i|\leq c_m$ which defines the domain to be 
searched at step 5 can be reduced. During the run of step 5 one may
update the quantity $d$ which is the actual minimum of the 
numbers $(\gamma_r^2/\sqrt{r}) \lVert C\rVert$ over the matrices $C$ enumerated
so far (here $r$ is the rank of $C$). From 
(\ref{alpha-bound}), 
Theorem \ref{kitconstant},  and (\ref{berge-martinet}) it follows that 
\begin{equation} \label{alpha-bound2}
|\alpha_i|\leq c_m \frac{\min \{d,\gamma'_n\} }
{\lVert\mathbf{b}_i\rVert}.
\end{equation}

\section{Two by two matrices over imaginary quadratic fields} 
   
Here we consider possible extensions of the improvements obtained 
over $\Q$ to other number fields. In general not much is known about tensor
products of lattices over general number fields. On the positive side,
Coulangeon \cite{Coulangeon} extended some of Kitaoka's results to imaginary
quadratic fields. 

Let  $\K$ denote an imaginary quadratic number field
$\Q(\sqrt{-d})$ where $d$ is a square-free positive integer. By $\OK$
we denote the ring of algebraic integers in $\K$:
$\OK=\Z 1+\Z \sqrt{-d}+\Z\frac{1+\sqrt{-d}}{2}$ if $d\equiv -1$ modulo $4$
and $\OK=\Z 1+\Z \sqrt{-d}$ otherwise. In the next discussion 
we consider $\K$ (and hence $\OK$)
to be embedded into $\C$.

Let $\Aalg$ be a central simple algebra of dimension $4$ over $\K$
isomorphic to $M_2(\K)$ and let $\Lambda$ be a maximal order in
$\Aalg$. We assume that we are given an embedding 
$\phi$ of $\Aalg$ into $M_2(\C)$. From the theory of central simple algebras 
over number fields we know (see Corollary 27.6 in Reiner \cite{Re}) that  
there exists a matrix $B\in M_2(\C)$ and a fractional ideal $I$
of $\OK$ such that 
$$B\phi(\Lambda)B^{-1}=\begin{pmatrix} \OK & I^{-1} \\ I & \OK
\end{pmatrix}.$$
In other words, there exists a full $\OK$-lattice $L$ in $\C^2$
(a finitely generated $\OK$-submodule of 
$\C^2$ that spans $\C^2$ as a linear space over $\C$)
such that $$\Lambda=\left\{A\in M_2(\C): AL\subseteq L \right\}.$$
In fact, for our specific $\Lambda$, we can take
$L=B\begin{pmatrix} \OK \\ I \end{pmatrix}$.

Let $\langle,\rangle$ stand for the standard Hermitian bilinear form
on $\C^2$. The linear extension of the mapping $u\otimes v\mapsto A_{u,v}$
where $A_{u,v}w=\langle v,w \rangle u$ gives an identification
of $M_2(\C)$ with $\C^2\otimes_\C \C^2$. By this identification, the
tensor square of the standard Euclidean norm of $\C^2$ becomes
the Frobenius norm of matrices. Also,
$\Lambda$ is identified with $L\otimes_\OK L^*$ where
$$L^*=\left \{u\in \C^2:\langle u,v \rangle\in L \mbox{~for every~}v\in L
\right\}$$
and $L\otimes_\OK L^*$ is just the additive subgroup of $\C^2\otimes_\C C^2$
spanned by the tensors of the form $u\otimes v$ where $u\in L$ and
$v\in L^*$.

\medskip
\noindent
{\bf Remark.} We can speak about the rank of an element 
$\mathbf x\in L\otimes_\OK L^*$ in two ways. One is the minimal positive
integer $r$ such that $\mathbf x$ can be written as 
\begin{equation} \label{tensor-rank}
\mathbf x=\sum_{i=1}^r \mathbf x_i\otimes \mathbf y_i,
\end{equation}
for some vectors $\mathbf x_i\in L$ and $\mathbf y_i\in L^*$. The other 
possible notion of rank is the rank of $\mathbf x$ as a matrix from
$M_2(\C)$. The two notions are not the same\footnote{It is not hard to show
that the two notions of rank  coincide if $\OK$ is a principal 
ideal ring.}. As an example\footnote{We thank  G\'eza K\'os for
suggesting this example.}, let $d=5$, and $\Lambda= M_2(\OK )$ and consider 
the matrix 
$$ C=\begin{pmatrix} 3 & 1+\sqrt{-5} \\
1-\sqrt{-5} & 2
\end{pmatrix} 
$$
from $\Lambda$. We have $\det C=0$, hence $C$ has rank 1 as a matrix 
from $M_2(\C)$. Moreover, using the fact that 3 and 2 are irreducible 
elements in $\OK$, we see that $C$ is not a decomposable tensor from
$\OK^2 \otimes _\OK \OK^2$. 

We shall use the term {\em rank} in the former sense. We note also, that  in the minimal
representation (\ref{tensor-rank}) the vectors $\mathbf x_i$ and $\mathbf
y_i$ are linearly independent over $\K$. For a proof we refer to Lemma 3.1
in \cite{Coulangeon}.

\medskip

Let $M$ be an $\OK$-submodule of $\C^2$ generated by two 
linearly independent vectors. The determinant of
$M$ is defined as the following Gram matrix
$$\det M=\det\begin{pmatrix}
\langle \mathbf v_1, \mathbf v_1 \rangle &
\langle \mathbf v_1,  \mathbf v_2 \rangle \\
\langle  \mathbf v_2,  \mathbf v_1 \rangle &
\langle  \mathbf v_2,  \mathbf v_2 \rangle
\end{pmatrix},
$$
where $ \mathbf v_1,  \mathbf v_2$ is any basis for $M$.

Following the notation of \cite{Coulangeon}, 
we denote by $\gamma_h(M)$ the quantity 
$$\lVert \mathbf v\rVert ^2/(\det M)^{\frac{1}{2}},$$
where $\mathbf v$ is a shortest nonzero vector from $M$.

Let $D$ be the discriminant of $\K$ (we have $D=d$ if $d\equiv 3$ modulo 4 and
$D=4d$ otherwise.)
It is known (last paragraph of Subsection 2.1 in \cite{Coulangeon})
that $\gamma_h(M)=\gamma(M)\sqrt{D}/2,$ and hence 

\begin{equation} \label{gamma-eq}
\gamma_h(M)=\gamma(M)\sqrt{D}/2 \leq 
\gamma_4\sqrt{D}/2=\sqrt{D/2},
\end{equation} 
where $\gamma(M)$ is the ratio of
$\lVert \mathbf v\rVert ^2$ and the fourth root of the determinant of $M$,
considered as a $\Z$-lattice of rank 4, and $\gamma_4=\sqrt 2$ is 
the Hermite constant.

\medskip

Let us define $r(\Lambda)$ as the ratio 
between the squared length of the shortest rank 1 element 
of $\phi(\Lambda)$ and that of the shortest rank 2 element 
in $\phi(\Lambda)$.

\begin{proposition} \label{ratio}
We have 
$$ r(\Lambda)\leq \frac{1}{2}\gamma_h(M)\gamma_h(M'),$$
where $M$ is an $\OK$-sublattice of $L$ generated by two linearly
independent vectors over $\K$ and
and $M'$ is an $\OK$-sublattice of $L^*$ generated by two linearly 
independent vectors over $\K$.
\end{proposition}

\begin{proof}
Indeed, let $\mathbf v$ and $\mathbf w$ be shortest nonzero vectors from $L$
and $L^*$, respectively. Also, let $\mathbf \omega$ be a shortest nonzero 
vector of
rank 2 from $\Lambda \cong L\otimes_\OK L^*$. We have then 

$$ r(\Lambda)=\frac{\lVert \mathbf v\rVert ^2\cdot \lvert \mathbf w\rVert ^2}
{\lVert \mathbf \omega \rVert ^2}.$$
Similarly to the rational case (\ref{lower}) we obtain 
$$
\lVert \mathbf \omega \rVert^2 \ge 2 (\det M)^{1/2}(\det M')^{1/2}
$$
for some sublattices $M\leq L$ and $M'\leq L^*$ which are spanned by  two 
linearly independent vectors 
over $\K$ (see Proposition 3.2 from \cite{Coulangeon} for the
details). 

Let $\mathbf v'$ and $\mathbf w'$ be shortest nonzero vectors from $M$ and
$M'$, respectively. Clearly we have 
$\lVert \mathbf v\rVert \leq \lVert \mathbf v' \rVert$, and 
$\lVert \mathbf w\rVert \leq \lVert \mathbf w' \rVert$.
By putting all these together we obtain 
$$ r(\Lambda)=\frac{\lVert \mathbf v\rVert ^2\cdot \lvert \mathbf w\rVert ^2}
{\lVert \mathbf \omega \rVert ^2}\leq 
\frac{\lVert \mathbf v'\rVert ^2\cdot \lvert \mathbf w'\rVert ^2}
{2 (\det M)^{1/2}(\det M')^{1/2}}=\frac 12 \gamma_h(M)\gamma_h(M').
$$
\end{proof}

For $d=1$ by (\ref{gamma-eq}) and Proposition \ref{ratio} we have
$r(\Lambda)\leq \frac 12 \sqrt 2\cdot \sqrt 2=1$. Similarly, for $d=3$ 
we find that
$$r(\Lambda)\leq \frac 12 \sqrt{\frac 32} \sqrt{\frac 32}=\frac 34<1.$$
We have obtained the following:

\begin{proposition} \label{imaginary-prop}
For $d=1$, at least one of the smallest element of
$\phi(\Lambda)$ with respect to the Frobenius norm has rank one.
For $d=3$ every smallest element of $\phi(\Lambda)$ has rank one. $\Box$
\end{proposition}

\medskip

The following example shows that over the Gaussian rationals it does 
indeed occur that a shortest nonzero element of $\Lambda$ has rank 2.
Let $d=1$ let $L$ be $\OK$-submodule of $\C^2$ generated
by $(1,0)^T$ and $(\frac{1}{\sqrt 2},\frac{i}{\sqrt 2})^T$
and let
$\Lambda=\left\{A\in M_2(\C): AL\subseteq L \right\}$.
Then 
$$\Lambda=\left\{
\frac{1+i}{2}\begin{pmatrix} a & b \\ c & e \end{pmatrix}:
\begin{array}{c}
a,b,c,e\in\OK, \\
a+c\equiv a+b \equiv b+e\equiv c+e 
\equiv 0 \mbox{~mod~}
(1+i),
\\
 a+b+c+e\equiv 0 \mbox{~mod~}2
\end{array}
\right\}
$$
and the identity matrix is one of the elements of $\Lambda$
having the smallest Frobenius norm.

\medskip

Next we outline a direct, elementary proof of inequality 
(\ref{gamma-eq}).  Let $\K=\Q(\sqrt{-d})$ be an imaginary quadratic
number field, $\OK$ be the maximal order of $\K$. Suppose further that 
$\OK$ is a principal ideal ring.

\begin{lemma}
Let $z$ be any complex number. Then there exists
an element $\alpha \in \OK$ such that 
$|z-\alpha|\leq \kappa$, where
$$\kappa=
\left\{  
\begin{array}{ll}
\frac{d+1}{4\sqrt d} & \mbox{$d\equiv 3$ mod $4$,}  \\
\frac{\sqrt{d+1}}{2} & \mbox{otherwise.}
\end{array}
\right. $$ 
\end{lemma}

\medskip

The proof is a simple argument from elementary geometry which we omit here.
For $d=1$, we have $\kappa=\frac{\sqrt 2}{2}$, for $d=2$,
$\kappa=\frac{\sqrt 3}{2}$, for  $d=3$, $\kappa=\frac{\sqrt 3}{3}$,
for $d=7$, $\kappa=\frac{2\sqrt 7}{7}$, for $d=11$, $\kappa=\frac{3}{\sqrt
11}$. In these cases $\kappa<1$. For for $d=5,6,10$ and for $d> 11$ we have 
$\kappa>1$. 
Let us define $\tau=\lfloor \kappa+1 \rfloor$. Then obviously 
we have $\kappa/\tau <1$.


\begin{proposition} Suppose that $\kappa<1$ holds, and 
let $M$ be an $\OK$-submodule of $\C^2$
generated by two linearly independent vectors over $\K$.
Then $$\gamma_h(M)\leq \frac{\tau}{\sqrt{1-\kappa^2}}.$$
\end{proposition}

\begin{proof}
Let $ \mathbf v,\mathbf w$ be a basis of $M$ such that $\mathbf v$ is a 
shortest nonzero vector
from $M$. Such a basis exists because $\OK$ is a principal ideal ring. 
Apply now the preceding lemma for $z= \tau \frac{\langle \mathbf v,\mathbf w
\rangle}{\langle \mathbf v,\mathbf v\rangle}$. There exists an 
$\alpha\in\OK$ be such that
$$\left| \frac{\langle \mathbf v,\mathbf w \rangle}{\langle \mathbf v, \mathbf v\rangle}
-\frac{\alpha}{\tau}\right|
\leq \frac{\kappa}{\tau}$$ 
and put $\mathbf w'=\mathbf w-\frac{\overline{\alpha}}{\tau}\mathbf v$. 
Then $\tau \mathbf w'\in M$
and hence 
$\lVert \mathbf w'\rVert \geq \frac{\lVert \mathbf v\rVert }{\tau}$. 
Furthermore,
$$ \left| \langle \mathbf v,\mathbf w'\rangle \right| = 
\left| \langle \mathbf v,\mathbf w-\frac{\overline{\alpha}}{\tau}
\mathbf v \rangle \right|=
\left| \langle \mathbf v,\mathbf w\rangle -\frac{\alpha}{\tau} \langle
\mathbf v,\mathbf v\rangle \right|
\leq \frac{\kappa}{\tau} \langle
\mathbf v, \mathbf v \rangle.$$
The Gram determinant does not change if we switch from the basis 
$\mathbf v,\mathbf w$ 
to  $\mathbf v,\mathbf w'$, hence  
\begin{eqnarray*}
\det M &=&\det\begin{pmatrix}\langle \mathbf v, \mathbf v \rangle &
\langle \mathbf v, \mathbf w' \rangle \\ 
\langle \mathbf w', \mathbf v \rangle &   
\langle \mathbf w', \mathbf w' \rangle \end{pmatrix}
=\langle \mathbf v, \mathbf v \rangle \langle \mathbf w', \mathbf w' \rangle-
|\langle \mathbf v, \mathbf w' \rangle|^2 \\
& \geq &
\left( \frac{1}{\tau^2}-\frac{\kappa^2}{\tau^2}\right) \lVert \mathbf v\rVert ^4.
\end{eqnarray*}
\end{proof}

For $d=1$ the Proposition gives $\gamma_h(M)\leq \sqrt{2}=\sqrt{D/2}$.
For $d=2$ we obtain $\gamma_h(M)\leq 2=\sqrt{D/2}$. 
For $d=3$ our bound is $\gamma_h(M)\leq \frac{\sqrt{3}}{\sqrt {2}}=
\sqrt{D/2}$. For $d=7$ the proposition gives
$\gamma_h(M)\leq \sqrt{\frac{7}{3}}<\sqrt{\frac{7}{2}}=\sqrt{\frac{D}{2}}$.
For these values of $d$ the ring $\OK$ is a principal ideal ring, 
hence we have proved (\ref{gamma-eq}).

\subsection*{The improved algorithm when $d=1$ or $d=3$}

We can achieve an improvement of the algorithm of Section 3 from
\cite{Ronyai} for $n=2$ and $d=1$ or $d=3$, i.e. for the case of two by two
matrices over the Gaussian rationals or over the Eisenstein rationals. 
These cases of the explicit isomorphism problem occur when one considers 
parametrization of Del Pezzo surfaces of degree 8, see Section 4 in 
\cite{Graaf}. Our method may present a viable alternative there 
to solving norm equations. 

Our improvement over the method of \cite{Ronyai}
is very similar to that of the algorithm over $\Q$. 
Suppose therefore that $\K$ is either $\Q(\sqrt{-1})$ or $\Q(\sqrt{-3})$,
and we have as input an algebra $\A$ over $\K$ specified by structure
constants. We assume that $\A \cong M_2(\K)$.  
The first four steps of the method in \cite{Ronyai} construct a 
maximal order $\Lambda $, an embedding $\phi:\Lambda \rightarrow M_2(\C)$,
 a $\Z$ linear embedding $\Phi$ of $\Lambda$
into $\R^8$ which maps an $y\in \Lambda$ to 
$$ \Phi(y):=(\Re \phi(y),\Im \phi(y))\in \R^8.$$
 The image $\Phi(\Lambda)$ is a   
full $\Z$ lattice in $\R^8$. Note that the (real) Euclidean norm $\lVert
\Phi(y)\rVert $ is the same as the Frobenius norm $\lVert \phi(y) \rVert$ 
inherited from $M_2(\C)$. 

Moreover, the first four steps return a $\Z$
basis $\mathbf b_1,\ldots ,\mathbf b_8$ of $\Phi(\Lambda)$ for which we have 
$$\lVert \mathbf b_1\rVert \cdot 
\lvert \mathbf b_2\rVert \cdots \lVert \mathbf
b_8\rVert \leq c_8 \det \Phi(\Lambda).$$

From this point on we can simplify the search for a zero divisor.  
By Proposition \ref{imaginary-prop} it suffices to enumerate\footnote{For
$d=3$ it suffices to find just one element with minimal norm.} the 
elements of $\Phi(\Lambda)$ which have minimal norm until a zero divisor in
$\Lambda$ is found.

To this end, we generate all integral linear combinations
$$\mathbf v =\sum_{i=1}^8  \alpha_i\mathbf{b}_i,$$
where $\alpha_i$ are integers,
$|\alpha_i|\le c_8$,
until a $\mathbf v$ is found for which the matrix $y\in \Lambda$ with
$\Phi(y)=\mathbf v$ is of rank one. The bound on the 
integers $\alpha_i$ follows from Lemma \ref{coefficients}, like in the
rational case.

\subsection*{Acknowledgment}

We are grateful to Josef Schicho for discussions on the subject.

\end{document}